\documentclass[12pt]{article}

\usepackage{amssymb} \usepackage{graphicx}
\usepackage[utf8x]{inputenc}
\usepackage{enumerate} \usepackage{amsfonts} \usepackage{amsmath}
\usepackage{stmaryrd} 
\usepackage{xspace}

\usepackage{hyperref}
\usepackage{amsthm}
\usepackage{color}
\usepackage{authblk}
\hypersetup{
  pdftitle = {Edep},
  pdfauthor = {J-.F. Raymond, I. Sau and D. M. Thilikos},
  colorlinks = true,
  linkcolor = black!30!red,
  citecolor = black!30!green
}

\newcommand{\ie}{\emph{i.e.}}

\newcommand{\floor}[1]{\left\lfloor#1\right\rfloor}
\newcommand{\ceil}[1]{\left\lceil#1\right\rceil}
\newcommand{\third}[1]{\frac{#1}{3}}
\newcommand{\paren}[1]{\left(#1\right)}
\newcommand{\acc}[1]{\left\{ #1 \right\}}

\newcommand{\N}{\mathbb{N}}
\newcommand{\R}{\mathbb{R}}
\newcommand{\intv}[2]{\left \llbracket #1, #2 \right \rrbracket}
\DeclareMathOperator{\tw}{\mathbf{tw}}

\DeclareMathOperator{\dg}{\mathrm{deg}} 

\DeclareMathOperator{\vertices}{V}
\DeclareMathOperator{\edges}{E}

\DeclareMathOperator{\neigh}{{N}} 

\newcommand{\desc}[1]{\mathrm{desc}_{#1}}
\newcommand{\induced}[2]{#1\!\left [#2 \right ]} 

\DeclareMathOperator{\maxdeg}{{\Delta}} 
\newcommand{\card}[1]{\left | #1 \right |} 
\newcommand{\subgraph}{\subseteq} 
\newcommand{\lminor}{\leq_\mathrm{m}} 
\newcommand{\contains}{\geq_\mathrm{m}} 

\newcommand{\marked}[1]{(#1)}

\newcommand{\pack}[1]{\mathbf{pack}^\mathbf{v}_{#1}}
\newcommand{\cover}[1]{\mathbf{cover}^\mathbf{v}_{#1}}
\newcommand{\edpack}[1]{{\mathbf{pack}^\mathbf{e}_{#1}}}
\newcommand{\edcover}[1]{{\mathbf{cover}^\mathbf{e}_{#1}}}

\DeclareMathOperator{\polylog}{polylog}

\renewcommand{\leq}{\leqslant} \renewcommand{\geq}{\geqslant}

\newtheorem{theorem}{Theorem}
\newtheorem{lemma}{Lemma}
\newtheorem{corollary}{Corollary}
\newtheorem{proposition}{Proposition}
\newtheorem{claimb}{Claim}

\theoremstyle{remark}
\newtheorem{remark}{Remark}

\theoremstyle{definition}

\def\claimb{$$\vcenter\bgroup\advance\hsize by -8em\noindent
\refstepcounter{claimb}\ignorespaces\it} \makeatletter
\def\endclaimb{\rm\egroup\leqno(\theclaimb)$$\global\@ignoretrue}
\makeatother

\newenvironment{proofclaim}[1][]{
\noindent \emph{Proof~#1. }{}{}{}} {\hfill$\Diamond$\vspace{1em}}

\usepackage{tikz}
\usetikzlibrary{decorations.pathreplacing}
\tikzstyle{every node} = [draw, circle, fill = black, minimum size = 4pt, inner sep = 0pt]
\tikzstyle{normal} = [draw=none, fill = none]

\begin{document}

\date{}

\title{An edge variant of the Erd{\H o}s-P{\'o}sa property\footnote{
This work was supported by the ANR project AGAPE
  (ANR-09-BLAN-0159) and the Languedoc-Roussillon Project ``Chercheur
  d'avenir'' KERNEL. The first author was partially supported by
  the (Polish) National Science Centre grant PRELUDIUM
  2013/11/N/ST6/02706. The third author was co-financed by the European
  Union (European Social Fund ESF) and Greek national funds through
  the Operational Program ``Education and Lifelong Learning'' of the
  National Strategic Reference Framework (NSRF) - Research Funding
  Program: ARISTEIA II. Email addresses: \texttt{jean-florent.raymond@mimuw.edu.pl}, \texttt{ignasi.sau@lirmm.fr}, and \texttt{sedthilk@thilikos.info}.}}

\author[1,2,3]{Jean-Florent Raymond}
\author[2]{Ignasi Sau}
\author[2,4]{Dimitrios M. Thilikos}

\affil[1]{Faculty of Mathematics, Informatics and Mechanics,
  University of Warsaw, Warsaw, Poland}

\affil[2]{AlGCo project team, CNRS, LIRMM, Montpellier, France}
\affil[3]{University of Montpellier, Montpellier, France}
\affil[4]{Department of Mathematics, National and Kapodistrian
  University of Athens, Athens, Greece}

\maketitle
\begin{abstract}
\noindent For every~$r\in \N$, we denote by $\theta_{r}$ the
multigraph with two vertices and $r$ parallel edges. Given a graph
$G$, we say that a subgraph $H$ of $G$ is {\em a model of
$\theta_{r}$ in $G$} if $H$ contains $\theta_{r}$ as a contraction.
We prove that the following edge variant of the Erdős-Pósa
property holds for every $r\geq 2$: if $G$ is a graph and $k$ is a
positive integer, then either $G$ contains a packing of $k$ mutually
edge-disjoint models of $\theta_{r}$, or it contains a set $S$ of
$f_r(k)$ edges such that $G\setminus S$ has no $\theta_{r}$-model, for both $f_r(k) = O(k²r³ \polylog kr)$ and
$f_r(k) = O(k^4r² \polylog kr).$


\end{abstract}

\section{Introduction}
\label{sec:intro}

Typically, an Erdős-Pósa property reveals relations between covering
and packing invariants in combinatorial structures. The origin of the
study of such properties comes from the Erdős-Pósa
Theorem~\cite{ErdosP65}, stating that there is a function $f:\N
\rightarrow \N$ such that for every $k \in \N$ and for every graph
$G,$ either $G$ contains $k$ vertex-disjoint cycles, or there is a set
$X$ of $f(k)$ vertices in $G$ meeting all cycles of $G$. In
particular, Erdős and Pósa proved this result for $f(k)=O(k\cdot \log
k)$.

An interesting line of research aims at extending Erdős-Pósa Theorem
for packings and coverings of more general graph structures. In this
direction, we say that a graph class ${\cal G}$ {\em satisfies the
Erdős-Pósa property} if there exists a function $f_{\cal G}: \N \to
\N$ such that, for every graph $G$ and every positive integer $k$,
either $G$ contains $k$ mutually vertex-disjoint subgraphs, each
isomorphic to a graph in ${\cal G}$, or it contains a set $S$ of
$f_{\cal G}(k)$ vertices meeting every subgraph of $G$ that is
isomorphic to a graph in ${\cal G}$.  When this property holds for a
class ${\cal G}$, we call the function $f_{\cal G}$ {\em the gap of
the Erdős-Pósa property for the class ${\cal G}$.} In this sense, the
classic Erdős-Pósa Theorem says that the class containing all cycles
satisfies the Erdős-Pósa property with gap $O(k \cdot \log k)$.

Given a graph $J$, we denote by ${\cal M}(J)$ the set of all graphs
containing $J$ as a contraction.  Robertson and Seymour proved the
following proposition, which in particular can be seen as an extension
of the Erdős-Pósa Theorem.

\begin{proposition}\label{thrs} Let $J$ be a graph. The class ${\cal
M}(J)$ satisfies the Erdős-Pósa property if and only if $J$ is
planar.
\end{proposition}

A proof of Proposition~\ref{thrs} appeared for the first time in
\cite{RobertsonS86GMV}. Another proof can be found in Diestel's
monograph~\cite[Corollary 12.4.10 and Exercise 40 of
Chapter~12]{Diestel05grap}. In view of Proposition~\ref{thrs}, it is
natural to try to derive good estimations of the gap function
$f_{{\cal M}(J)}$ in the case where $J$ is a planar graph.  In this
direction, the recent breakthrough results of Chekuri and Chuzhoy
imply that $f_{{\cal M}(J)}(k)=k\cdot
\polylog k$~\cite{ChekuriC13larg} when $J$ is a planar graph and, even more, that $f_{{\cal
M}(J)}=(k+|V(J)|)^{O(1)}$~\cite{ChekuriC13poly}.  Before~this, the
best known estimation of the gap for planar graphs was exponential,
namely $f_{{\cal M}(J)}(k) = 2^{O(k \log k)}$, and could be deduced
from~\cite{LeafS12sube} using the proof
of~\cite{RobertsonS86GMV}. Moreover, some improved polynomial gaps
have been proven for particular instantiations of the graph $J$
(see~\cite{FominLMPS13quad,RaymondT13poly,RaymondT13lowp,FioriniJS13,FioriniJW12excl,
FioriniHJ13}). Another direction is to add restrictions on the graphs
$G$ that we consider, which usually allows to optimize the gap
$f_{{\cal M}(J)}$.  In this direction, it is known that $f_{{\cal
M}(J)}=O(k)$ in the case where graphs are restricted to some
non-trivial minor-closed class~\cite{FominST11stre}.
\medskip

We consider the edge counterpart of the Erdős-Pósa property, where
packings are edge-disjoint (instead of vertex-disjoint) and coverings
contain edges instead of vertices. We say that a graph class ${\cal
G}$ satisfies {\em the edge variant of the Erdős-Pósa property} if
there exists a function $f_{\cal G}$ such that, for every graph $G$
and every positive integer $k$, either $G$ contains $k$ mutually
edge-disjoint subgraphs, each isomorphic to a graph in ${\cal G}$, or
it contains a set $X$ of $f_{\cal G}(k)$ edges meeting every subgraph of
$G$ that is isomorphic to a graph in ${\cal G}$. Recently, the edge
variant of the Erdős-Pósa property was proved
in~\cite{KawarabayashiK12} for 4-edge-connected graphs in the case
where ${\cal G}$ contains all odd cycles.  \medskip

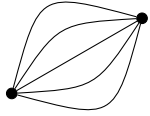
\begin{figure}[h]
 \centering
 \begin{tikzpicture}[rotate=30]
  \node (x) at (0,0) {};
  \node (y) at (2,0) {};
  \draw (x) .. controls (1, 1) and (1, 1) .. (y);
  \draw (x) .. controls (1, -1) and (1, -1) .. (y);
  \draw (x) .. controls (1, .5) and (1, .5) .. (y);
  \draw (x) .. controls (1, -.5) and (1, -.5) .. (y);
  \draw (x) -- (y);
 \end{tikzpicture}
 \caption{The graph $\theta_5.$}
 \label{fig:t5}
\end{figure}

In this paper we concentrate on the case where ${\cal G}={\cal M}(J)$
for some graph $J$.  We find it an interesting question whether an
edge-analogue of Proposition~\ref{thrs} exists or not.  To our
knowledge, the only case for which ${\cal M}(J)$ satisfies the edge
variant of the Erdős-Pósa property is when $J=K_{3}$, \ie~when the
class of graphs ${\cal G}$ contains all cycles. This result is the
edge-counterpart of the Erdős-Pósa Theorem and appears as a (hard)
exercise in~\cite[Exercise 23 of Chapter 7]{Diestel05grap}. For every
$r\geq 2$, let $\theta_{r}$ be the graph containing two vertices and
$r$ multiple edges between them (see Figure~\ref{fig:t5}). The results
of this paper can be stated as follows:

\begin{theorem}
\label{main} The edge variant of the Erdős-Pósa property holds for
${\cal M}(\theta_{r})$ with gap $f_{{\cal
M}(\theta_{r})},$ with
\begin{align*}
  f_{{\cal M}(\theta_{r})}(k) = O(k²r³ \polylog kr)\quad\text{and}\quad f_{{\cal M}(\theta_{r})}(k) = O(k^4r² \polylog kr).
\end{align*}
\end{theorem}

Theorem~\ref{main} is the edge-counterpart of the main result
of~\cite{FominLMPS13quad}. The proof is presented in
Section~\ref{sec:eep} and contains three main ingredients. The first
is a reduction of the problem to graphs of bounded degree, presented
in Section \ref{sec:deltabound}. The second is an application of recent
results of~\cite{ChekuriC13larg} to obtain bounds on the treewidth of
the graphs we consider (Section~\ref{sec:bound-tw}) and the last is an
extension of the techniques in~\cite{FominST11stre} fitting our needs,
which is presented in
Section~\ref{sec:vert-edg}. Section~\ref{sec:secpr} contains
definitions and preliminary results and Section~\ref{sec:fr} discusses
further research about the problem investigated in this paper.

\section{Definitions and preliminaries}
\label{sec:secpr}

For any graph~$G$, $\vertices(G)$ (respectively $\edges(G)$) denotes the \emph{set
of vertices} (respectively \emph{edges}) of~$G$. Even when dealing with
multigraphs (\ie~graphs where more than one edge is allowed between
two vertices) we will use the term \emph{graph}.
A graph~$G'$ is a \emph{subgraph} of a graph~$G$ if $\vertices(G')
\subseteq \vertices(G)$ and $\edges(G') \subseteq \edges(G)$, and we
denote this by~$G' \subgraph G$.
If~$X$ is a subset of~$\vertices(G)$ (respectively~$\edges(G)$), we denote by~$\induced{G}{X}$ the
\emph{subgraph of~$G$ induced by~$X$}, \ie~the graph with vertex set
$X$ (respectively~$\cup_{e \in X} e$) 
and edge set $\acc{\{x,y\} \in \edges(G),\ x \in X\ \mathrm{and}\
 y \in X}$ (respectively $X$). If $S$ is a subset of vertices or edges of a graph $G$,
the graph $G \setminus S$ is the graph obtained from $G$ after
the removal of the elements of $S.$ For every vertex~$v \in\vertices(G)$ the
 neighborhood of~$v$ in~$G,$ denoted by~$\neigh_G(v),$ is the subset
 of vertices that are adjacent to~$v$, and its size is called the
 \emph{degree} of~$v$ in~$G,$ written~$\dg_G(v).$ The maximum
 degree~$\maxdeg(G)$ of a graph~$G$ is the maximum value taken
 by~$\dg_G$ over~$\vertices(G).$
 Given a non-negative integer $k$, a triple $(V_1, S,
 V_2)$ is called a {\em $k$-separation triple} of a graph $G$ if
 $\card{S} \leq k$ and $\{V_1, S, V_2\}$ is a partition of
 $\vertices(G)$ such that there is no edge between a vertex of $V_1$
 and a vertex of $V_2.$
 Unless otherwise stated, logarithms are binary.
 For any two integers $a,b$ such that $a \leq b,$ the notation
 $\intv{a}{b}$ stands for the set of integers $\{a, a+1, \dots, b\}.$
 In a tree $T$, rooted at a vertex $r \in \vertices(T),$ a vertex $u \in
 \vertices(T)$ is said to be a \emph{descendant} of a vertex $v \neq
 u$ if the path in $T$ from $r$ to $u$ contains~$v.$ The set of
 descendants of~$v$ is denoted by $\desc{T}(v).$ A graph is
 \emph{biconnected} if the removal of any vertex leaves the
 graph connected, and a biconnected component of a graph is a maximal
 biconnected subgraph.

\paragraph{Minors and models} In a graph~$G,$ a \emph{contraction} of
an edge $e=\{u,v\} \in \edges(G)$ is the operation that removes $e$
from $G$ and identifies the vertices~$u$ and~$v.$ In this paper, we
keep multiple edges that may appear between two vertices after a
contraction (for instance, contracting an edge in a triangle gives a
graph with two vertices connected by two edges).
For any graph~$J$, let~$\mathcal{M}(J)$ denote the class of \emph{contraction
 models} (models for short) of $J$, \ie~the class of graphs that can
be contracted to $J.$ We say that a graph $J$ is \emph{minor} of a
graph $G$ (denoted by $J \lminor G$) if a subgraph of $G$ is a model
of $J$ (\emph{$J$-model} for short), or, equivalently, if $J$ can be
obtained from $G$ by a series of vertex deletions, edge deletions, and
edge contractions.

\paragraph{Packings and coverings} Let $G$ and $J$ be graphs.
We denote by $\pack{J}(G)$ the maximum number of vertex-disjoint models
of $J$ in~$G$ and by $\cover{J}(G)$ the minimum size of a subset $S
\subseteq \vertices(G)$ (called $J$-\emph{vertex-hitting set}) that meets
the vertex sets of all models of $J$ in $G.$ These invariants are
widely studied in the context of the classic Erdős-Pósa property.

Similarly, we write $\edpack{J}(G)$ for the maximum number of
edge-disjoint models of $J$ in~$G$ and $\edcover{J}(G)$ for the minimum
size of a subset $S \subseteq \edges(G)$ (called
$J$-\emph{edge-hitting set}) that meets the edge sets of all models of
$J$ in $G.$
Obviously, for every two graphs $G$ and $J$, the following inequality
holds:
\[
\edpack{J}(G) \leq \edcover{J}(G).
\]
A graph $J$ is said to satisfy the \emph{(vertex-)Erdős-Pósa property for minors}
(\emph{vertex-Erdős-Pósa property} for short) if there is a function
$f_J \colon \N \to \N,$ called \emph{vertex-Erdős-Pósa gap} of $J$, such that for every graph $G$, the following holds:
\[
\cover{J}(G) \leq f_J(\pack{J}(G)).
\]

The research of this paper is motivated by the course of detecting
 graphs $J$ for which there is
a function $h_J \colon \N \to \N$ satisfying the following inequality for every graph $G$:
\begin{eqnarray}\label{eqn:edep}
\edcover{J}(G) \leq h_J(\edpack{J}(G)).
\end{eqnarray}
Such graphs are said to satisfy the \emph{edge variant of the Erdős-Pósa property for
minors} (or, in short, the \emph{edge-Erdős-Pósa property}) and the
function $h_{J}$ is called the \emph{gap} of the
edge-Erdős-Pósa property for~$J$ (\emph{edge-Erdős-Pósa gap} for
short). This definition is an edge-counterpart to the existing Erdős-Pósa
property and (vertex-)Erdős-Pósa gap.

\paragraph{Treewidth}
A \emph{tree decomposition} of a graph~$G$
is a pair~$(T,\mathcal{V})$ where $T$ is a tree and
$\mathcal{V}$ a family $(V_t)_{t \in \vertices(T)}$ of
subsets of $\vertices(G)$ (called \emph{bags}) indexed by
the vertices of $T$ and such that
 \begin{enumerate}[(i)]
 \item $\bigcup_{t \in \vertices(T)} V_t = \vertices(G)$;
 \item for every edge~$e$ of~$G$ there is an element of~$\mathcal{V}$
containing both endpoints of~$e$; and
 \item for every~$v \in \vertices(G)$, the subgraph of~${T}$
induced by $\{t \in \vertices(T)\mid {v \in V_t}\}$ is connected.
 \end{enumerate}

The \emph{width} of a tree decomposition~${T}$ is defined as
$\max_{t \in \vertices(T)}~{\card{V_t} - 1}$ (that is, the maximum
size of a bag minus one). The \emph{treewidth} of~$G$,
written~$\tw(G)$, is the minimum width of any of its tree
decompositions.\\

 A tree decomposition $(T,\mathcal{V})$ of a graph $G$ is said to be a \emph{nice} tree
 decomposition if
 \begin{enumerate}[(i)]
 \item every vertex of $T$ has degree at most 3;
 \item $T$ is rooted at one of its vertices $r$ whose bag is empty ($V_r = \emptyset$); and
 \item every vertex $t$ of $T$ is
  \begin{itemize}
  \item either a \emph{base node}, \ie~a leaf of $T$ whose
   bag is empty ($V_t = \emptyset$) and different from the root;

  \item or an \emph{introduce node}, \ie~a vertex with only one
   child $t'$ such that $V_{t} = V_{t'} \cup \{u\}$ for some $u \in
   \vertices(G)$;

  \item or a \emph{forget node}, \ie~a vertex with only one
   child $t'$ such that $V_{t'} = V_{t} \cup \{u\}$ for some $u \in
   \vertices(G)$;

  \item or a \emph{join node}, \ie~a vertex with two children $t_1$ and
   $t_2$ such that $V_t = V_{t_1} = V_{t_2}$.
  \end{itemize}
 \end{enumerate}
It is known that every graph has an optimal tree decomposition which
is nice \cite{Kloks94}.

\paragraph{The graph $\theta_r$ and the Erdős-Pósa property}
The vertex-Erdős-Pósa property of $\theta_r$ received
some attention, in particular in~\cite{FominLMPS13quad, FioriniJS13,
  JoretPSST11}. For instance the main result of \cite{FominLMPS13quad}
is the following estimation of the vertex-Erdős-Pósa gap
for~$\theta_r$.
\begin{proposition}[\cite{FominLMPS13quad}]\label{p:vep}
  For every positive integer $r$, $\theta_r$ has the vertex-Erdős-Pósa property
  with gap~$O(k²).$
\end{proposition}
However in this estimation the dependency in terms of $r$ is hidden in the multiplicative constant of the Big-O notation.
By a careful analysis of the size of a $\theta_r$-hitting
set presented in~\cite{FominLMPS13quad} (c.f.\
Lemma~\ref{l:small_tw-new}), the estimation of the gap of
Proposition~\ref{p:vep} can be made quadratic in both $k$ and~$r.$
From this, we can derive an $O(k³r³)$ edge-Erdős-Pósa gap for $\theta_r$
(Corollary~\ref{c:cubic}) by using our Lemma~\ref{l:transl} that makes
possible to translate a $\theta_r$-vertex-hitting set
into a $\theta_r$-edge-hitting~set.

However, Theorem~\ref{main} gives better estimations of this gap,
either in $k$ or in~$r.$

\paragraph{Patterns in graphs of big treewidth}
In the following section, we will use several propositions asserting
that every graph $G$ of treewidth at least $c_H$ contains some fixed
graph $H$ as a minor, where the constant $c_H$ depends on~$H.$ For
instance, we will show in Lemma~\ref{l:small_tw-new} a simple relation
between the constant $c_{k \cdot \theta_r}$ and the vertex-Erdős-Pósa
gap for~$\theta_r.$ These propositions as
stated~thereafter.

\begin{proposition}[{\cite[Lemma~3.2]{ReedW08poly}}] \label{p:one_pumpkin}
  For every integer $r \geq 1$ and graph $G,$ if $\tw(G) \geq 2r-1$
  then $G$ contains a $\theta_r$-model.
\end{proposition}

\begin{proposition}[{\cite[Lemma~1]{FominLMPS13quad}}, see
  also~\cite{BirmeleBR07bram}] \label{p:big_tw}
 Let $k$ and $r$ be two positive integers. For every graph $G$, if
 $\tw(G) \geq 2k^2r^2$ then $G$ contains at least $k$ vertex-disjoint
 models of $\theta_{r}.$
\end{proposition}

 \begin{proposition}[{\cite[Theorem~1.1]{ChekuriC13larg}}]\label{p:cc_h}
   There is a function $f_{\mathrm{Prop\ref{p:cc_h}}}(t) = O(\polylog
   t)$ such that, for every graph $G$ and every positive integers $h$
   and $p,$ if $hp^2 \leq \frac{\tw(G)}{f_{\mathrm{Prop\ref{p:cc_h}}}(\tw(G))},$ there is a
   partition $G_1, \dots, G_h$ of $G$ into vertex-disjoint subgraphs
   such that $\tw(G_i) \geq p$ for each $i\in \intv{1}{h}.$
 \end{proposition}

 \begin{proposition}[{\cite[Theorem~1.2]{ChekuriC13larg}}]\label{p:cc_r}
   There is a function $f_{\mathrm{Prop\ref{p:cc_r}}}(t) = O(\polylog
   t)$ such that, for every graph $G$ and every positive integers $h$
   and $p,$ if $h^3p \leq \frac{\tw(G)}{f_{\mathrm{Prop\ref{p:cc_r}}}(\tw(G))}$ then there is a partition
   $G_1, \dots, G_h$ of $G$ into vertex-disjoint subgraphs such that
   $\tw(G_i) \geq p$ for each $i\in \intv{1}{h}.$
 \end{proposition}

\section{The edge-\texorpdfstring{Erdős-Pósa}{Erdös-Posa} property
 for \texorpdfstring{graphs $\theta_r$}{pumpkin graphs}}
\label{sec:eep}

\subsection{Bounding the degree}
\label{sec:deltabound}

In the sequel, we deal with graphs in which some vertices
are marked.
If $G$ is a graph and $m \colon \vertices(G) \to \{0,1\}$ is a
function, we say that $\marked{G, m}$ is a graph \emph{marked} by $m.$
A vertex $v$ of $G$ such that $m(v) = 1$ is said to be~\emph{marked}.
We denote by $\mu$ the function that, given a graph,
returns its number of marked vertices.
We now define an $r$-good partition.
Given a positive integer $r$, a marked tree $\marked{T, m}$ is said to have an \emph{$r$-good partition
  of root $v$} if there is a pair~$\paren{\marked{T_1, m_1},
  \marked{T_2, m_2}}$ of marked trees such that:
 \begin{enumerate}[(i)]
 \item $T_1$ and $T_2$ are subtrees of $T$ such that $(\edges(T_1),
  \edges(T_2))$ is a partition of $\edges(T)$;\label{e:part}
 \item $r \leq \mu\paren{\marked{T_1, m_1}} \leq 2r$; \label{e:1lt}
 \item $v \in \vertices(T_2)$; and \label{e:2lt}
 \item every vertex that is marked in $\marked{T, m}$ is either
  marked in $\marked{T_1, m_1}$
  or marked in $\marked{T_2, m_2}$, but not in both. In other words,
  for every $u \in \vertices(T),$
  \begin{itemize}
  \item if $v \in \vertices(T_1) \cap \vertices(T_2)$ then $m(v) = 1
   \Leftrightarrow m_1(v) = 1\ \text{or}\ m_2(v) = 1$ but not both;
 \item otherwise, let $i \in \{1,2\}$ be the integer such that $v \in
   \vertices(T_i)$. Then we have $m(v) = m_i(v)$.
  \end{itemize} \label{e:mm}
 \end{enumerate}

 We remark that because of~(\ref{e:mm}), $\mu(T) = \mu(T_1) + \mu(T_2).$
 If for every $v \in \vertices(T),$ $\marked{T, m}$ has an $r$-good partition of
 root $v$, then $T$ is said to have an \emph{$r$-good partition}.
 Examples of a marked tree and of a good partition are given in Figure~\ref{fig:parti}.

 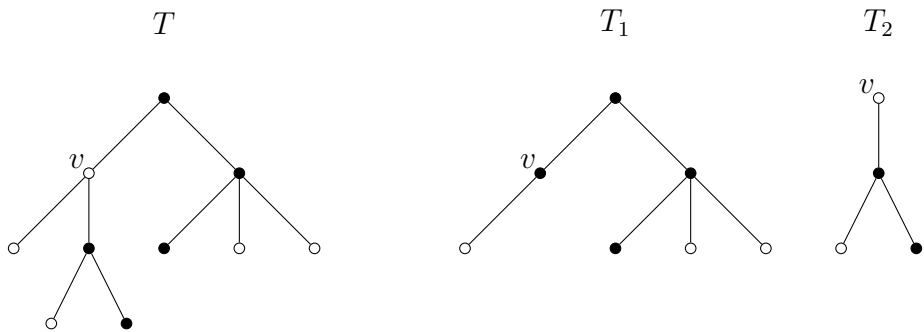
\begin{figure}[h]
   \centering
   \begin{tikzpicture}
     \begin{scope}
       \node[normal] at (0,1) {$T$};
       \node (n0) at (0,0) {};
       \node[fill = white, label=135:$v$] (n1) at (-1,-1) {};
       \node (n2) at (1,-1) {};
       \node[fill = white] (n3) at (-2,-2) {};
       \node (n4) at (-1, -2) {};
       \node (n5) at (0, -2) {};
       \node[fill = white] (n6) at (1, -2) {};
       \node[fill = white]  (n7) at (2, -2) {};
       \node[fill = white] (n8) at (-1.5, -3) {};
       \node (n9) at (-0.5, -3) {};
       \draw (n3) -- (n1) -- (n0) -- (n2) -- (n5) (n6) -- (n2) -- (n7) (n1) -- (n4) -- (n8) (n4) -- (n9);
     \end{scope}
     \begin{scope}[xshift = 6cm]
       \node[normal] at (0,1) {$T_1$};
       \node (n0) at (0,0) {};
       \node[label=135:$v$] (n1) at (-1,-1) {};
       \node (n2) at (1,-1) {};
       \node[fill = white] (n3) at (-2,-2) {};
       \node (n5) at (0, -2) {};
       \node[fill = white] (n6) at (1, -2) {};
       \node[fill = white] (n7) at (2, -2) {};
       \draw (n3) -- (n1) -- (n0) -- (n2) -- (n5) (n6) -- (n2) -- (n7);
     \end{scope}
     \begin{scope}[xshift = 10.5cm, yshift = 1cm]
       \node[normal] at (-1,0) {$T_2$};
       \node[fill = white, label=135:$v$] (n1) at (-1,-1) {};
       \node (n4) at (-1, -2) {};
       \node[fill = white] (n8) at (-1.5, -3) {};
       \node (n9) at (-0.5, -3) {};
       \draw (n1) -- (n4) -- (n8) (n4) -- (n9);
     \end{scope}
   \end{tikzpicture}
   \caption{A marked tree $T$ with $\mu(T)=5$ and a $3$-good partition $(T_1,T_2)$ of root $v$. Marked vertices appear in~white.}
   \label{fig:parti}
 \end{figure}

\begin{lemma}\label{l:rgp}
For every integer $r>0$ and every marked tree $\marked{T, m}$, if
$\mu(T) \geq 2r$ then $\marked{T, m}$ has an $r$-good partition.
\end{lemma}

\begin{proof}
Let $r>0$ be an integer. We prove this lemma by induction on the size of the tree.

\noindent \emph{Base case:} $\card{\vertices(T)} = 0$. Since $2r \geq 2 > \card{\vertices(T)},$ $T$ does not have $2r$ marked
 vertices and we are done.

\noindent \emph{Induction step:}
Assume that for every integer $n' < n,$ every marked tree $\marked{T',
 m'}$ on $n'$
vertices and satisfying $\mu(\marked{T', m'}) \geq 2r$ has an $r$-good partition
(induction hypothesis).

Let us prove that every marked tree on $n$ vertices has a $r$-good partition
if it has at least $2r$ marked vertices.
Let $\marked{T, m}$ be a tree on $n$ vertices and let $v$ be a vertex of $T.$ We
assume that $\mu(\marked{T, m}) \geq 2r.$ We distinguish two cases.

\begin{itemize}
\item $\mu(\marked{T, m}) = 2r$:

 Let $T_1 = T,$ let $m_1 = m,$ let $T_2 = (\{v\}, \emptyset)$, and
 let $m_2 \colon \vertices(T_2) \to \{0,1\}$ be the function equal to
 0 on every vertex of $T_2.$
 Remark that $(\edges(T_1), \edges(T_2)) = (\edges(T), \emptyset)$
 is a partition of $\edges(T)$, $T_2$ contains $v$, and as $\marked{T,
  m}$ contains
 (exactly) $2r$ marked vertices, so does $\marked{T_1, m}.$ Consequently
 $(\marked{T_1, m_1},\marked{T_2, m_2})$ is an $r$-good partition of
 $\marked{T, m}.$

\item $\mu(\marked{T, m}) > 2r$:

We distinguish different cases depending on the degree of the root~$v$ in $T.$
\begin{itemize}
\item[Case 1:]~$\dg(v) = 1$.

Let $u$ be the neighbor of $v$ in $T,$ let $T' = T \setminus \{v\}$,
and $m' = m_{|\vertices(T')}.$
Remark that $\mu(\marked{T',m'}) \geq 2r$ and $\card{\vertices(T')} =
\card{\vertices(T)} - 1.$ By induction hypothesis, $\marked{T', m'}$ has
an $r$-good partition $(\marked{T'_1, m_1'}, \marked{T'_2, m_2'})$ of root
$u.$ We extend it to $T$ by setting $T_1 = T'_1$, $m_1 = m_1'$, $T_2
= \paren{\vertices(T'_2) \cup \{v\}, \edges(T'_2) \cup \{v,u\}}$, and~$m_2 = m_2'$.
Notice that $T_2$ contains $v.$ As the subtree $T'_1$ contains at
least $r$ and at most $2r$ marked vertices (induction hypothesis), so
does $T_1$. Also, remark that $(\edges(T_1), \edges(T_2))$ is a
partition of $\edges(T)$ and that since $u \in T'_2,$ the graph $T_2$
is connected. Therefore $((T_1, m_1), (T_2, m_2))$ is an $r$-good partition
of~$T$.

\item[Case 2:]~$\dg(v) = d > 1.$

Let $u_1, \dots, u_d$ be the neighbors of $v$ in $T$ and for every $i
\in \intv{1}{d},$ let $C_i$ be the connected component of $T \setminus
\{v\}$ that contains $u_i.$ We also define, for every $i \in \intv{1}{d},$ the
restricted marking function $w_i = m_{|V(C_i)}.$

\begin{enumerate}[{Subcase} (a):]
\item there exists $i \in \intv{1}{d}$ such that $\mu(\marked{C_i, w_i}) > 2r.$

 Let $T' = (\vertices(C_i) \cup\{v\}, \edges(C_i) \cup \{u,v\})$ and
 let $m' = m_{|\vertices(T')}.$
 Remark that $\card{\vertices(T')} < \card{\vertices(T)}$ and
 $\mu(\marked{T', m'}) > 2r.$
 According to the induction hypothesis, $\marked{T', m'}$ has an $r$-good partition
 $(\marked{T_1', m_1'}, \marked{T_2', m_2'})$ of root $v$. Similarly as
 before, we can extend it into an $r$-good partition
 $(\marked{T_1, m_1}, \marked{T_2, m_2})$ of $\marked{T, m}.$ This is
 done by setting:
 \begin{align*}
  T_1 &= T'_1,\\
  m_1 &= m_1',\\
  T_2 &= (\vertices(T'_2) \cup (\vertices(G) \setminus \vertices(C_i)), \edges(G) \setminus \edges(T'_1)), \text{ and}\\
  m_2 &:\left \{
   \begin{array}{l}
    u \mapsto m_2'(u)\quad\text{if}\ u \in \vertices(C_i)\cup \{v\} \\
    u \mapsto m(u)\quad\text{otherwise}.
   \end{array}
\right .
 \end{align*}
As $(\marked{T_1', m_1'}, \marked{T_2', m_2'})$ is an $r$-good
partition of root $v$, $v \in \vertices(T_2')$ and therefore
$T_2$ is connected.

\item there exists $i \in \intv{1}{d}$ such that $r \leq \mu(\marked{C_i, w_i}) \leq 2r.$

 Let $T_1 = C_i$ and $T_2 = \induced{T}{\edges(T) \setminus
  \edges(T_1)}.$ In this case, $(\edges(T_1), \edges(T_2))$ is a partition
 of $\edges(T)$ and $T_2$ is connected since it contains $v$, the
 vertex which is adjacent to the $C_j$'s. Thus, if we set $m_1 =
 m_{|V(T_1)}$ and $m_2 = m_{|V(T_2)}$, $(\marked{T_1, m_1},
 \marked{T_2, m_2})$ is an $r$-good partition of $\marked{T, m}.$

\item for all $i \in \intv{1}{d}, \mu(\marked{C_i, w_i}) < r.$

 Let $j = \min \acc{j \in \intv{2}{d},\ \sum_{i = 1}^j \mu(\marked{C_i, w_i}) \geq
  r}$. Such value exists since $\mu((T,m)) > 2r$. We then set:
\begin{align*}
 T_1 &= (\cup_{i \in \intv{1}{j}} \vertices(C_i) \cup
 \{v\}, \cup_{i \in \intv{1}{j}} (\edges(C_i) \cup \{v, u_i\})),\\
 m_1 &: \left \{
  \begin{array}{l}
   v \mapsto 0\\
   u \in \vertices(T_1) \setminus \{v\} \mapsto m(u)
  \end{array}
 \right .,\\
 T_2 &= \induced{T}{\edges(T) \setminus \edges(T_1)}, \text{ and}\\
 m_2 &= m_{|V(T_2)}.
\end{align*}

By definition of $j,$
$\mu(\marked{T_1, m_1}) \geq r$ and as for every $i \in
\intv{1}{d},$
$\mu(\marked{C_i, w_i}) < r$ we also have
$\mu(\marked{T_1, m_1}) < 2r.$ As before, the pair
$(\marked{T_1, m_1}, \marked{T_2, m_2})$
is an $r$-good partition of $\marked{T, m}.$
\end{enumerate}
\end{itemize}
\end{itemize}

In conclusion, we proved by induction that for every integer $r,$ every
tree having at least $2r$ marked vertices has an $r$-good partition.
\end{proof}

In the sequel we will deal with packings of the graph $\theta_{r},$
for $r>1.$ The following remark is important.
\begin{remark}\label{r:2co}
 If $G$ is not biconnected, the number of edge-disjoint models of
 $\theta_{r}$ in $G$ is equal to the sum of the number of edge-disjoint models of
 $\theta_{r}$ in every biconnected component of $G.$
 This enables us to treat biconnected components separately.
\end{remark}

\begin{lemma}\label{l:bigdg}
 Let $k>0, r>0$ be two integers, and let $G$ be a biconnected graph
 with $\maxdeg(G) \geq 2kr$. Then $\edpack{\theta_r}(G) \geq k$.
\end{lemma}

\begin{proof}
As $G$ is biconnected, the removal of a vertex $v$ of maximum degree
gives a connected graph. Let $T$ be a minimal tree of $G \setminus
\{v\}$ spanning the neighborhood $\neigh_G(v)$ of $v.$ We mark the
vertices of $T$ that are elements of $\neigh_G(v)$: this gives the
marking function $m$ for $T.$
Let us prove by induction on $k$ that $\marked{T, m}$ has $k$
edge-disjoint marked subtrees $\marked{T_1, m_1}, \dots, \marked{T_k,
 m_k},$ each containing at least $r$ marked vertices. If we do so,
then we are done because $\acc{\{v\}, T_i}_{i \in \intv{1}{k}}$ is a
collection of $k$ edge-disjoint $\theta_r$ models. In fact, as for
every $i \in \intv{1}{k},$ $T_i$ contains $r' \geq r$ vertices
adjacent to $v$ in $G,$ contracting the edges of $T_i$ in
$\induced{G}{\{v\} \cup \vertices(T_i)}$ gives the graph
$\theta_{r'}.$ Let $r >0$ be an integer.

\noindent \emph{Base case $k = 1$:} Clear.

\noindent \emph{Induction step $k > 1$:} Assume that for every $k' <
k,$ every tree with at least $2k'r$ vertices marked has $k'$
edge-disjoint subtrees, each with at least $r$ marked vertices.
Let $\marked{T, m}$ be a marked tree such that $\mu(\marked{T, m}) \geq 2kr.$
According to Lemma~\ref{l:rgp}, $\marked{T, m}$ has an $r$-good
partition $(\marked{T_1, m_1}, \marked{T'_1, m_1'})$ such that $r \leq
\mu(\marked{T_1, m_1}) \leq 2r$ and $\mu(\marked{T'_1, m_1'}) =
\mu(\marked{T, m}) - \mu(\marked{T_1, m_1}) \geq 2(k-1)r.$
By induction hypothesis, $\marked{T'_1, m_1'}$ has $k-1$ edge-disjoint marked
subtrees $\marked{T_2, m_2}, \dots, \marked{T_k, m_k}$ each containing at least $r$ marked
vertices. Remark that as all these trees are subgraphs of $T'_1,$ which
is edge-disjoint from $T_1$ in $T$, they are edge-disjoint from $T_1$
as well. Consequently, $\marked{T_1, m_1},$
$ \marked{T_2, m_2}, \dots,$
$\marked{T_k, m_k}$ is the family of edge-disjoint subtrees we were looking~for.
\end{proof}

\subsection{Bounding the treewidth}
\label{sec:bound-tw}

\begin{lemma}\label{l:bound1}
  There is a function $h_r(k) = O(k r^2 \polylog kr)$ such that for every
  positive integers $k$ and $r$ and every graph $G$, if $\tw(G) \geq h_r(k)$,
then $\pack{\theta_r}(G) \geq k$.
\end{lemma}

\begin{proof}
  Let $G$ be a graph and $k,r$ be two positive integers. By
  Proposition~\ref{p:big_tw}, if $\tw(G) \geq 2k²r²,$ then $G$
  contains $k$ vertex-disjoint models of $\theta_r.$ Therefore, we
  only have to consider the case where $\tw(G) < 2k²r².$
  
  As $f_{\mathrm{Prop\ref{p:cc_h}}}(t) = O(\polylog t)$ (\textit{cf.} Proposition~\ref{p:cc_h} for
  the definition of $f_{\mathrm{Prop\ref{p:cc_h}}}$), there are three positive
  reals $t_0$, $A\geq 1$, and $\alpha \geq 1$ such that for every real $t \geq t_0$
  we have $f_{\mathrm{Prop\ref{p:cc_h}}}(t) \leq A
  \log^\alpha(t)$. Let $B = \max(0, \max_{i \in \intv{1}{\ceil{t_0}}} f_{\mathrm{Prop\ref{p:cc_h}}}(i))$ and
  observe that for every positive integer $i$ we have $f_{\mathrm{Prop\ref{p:cc_h}}}(i) \leq A
  \log^\alpha(i) + B$.

  Let $h_r(k) = k(2r)^2 \cdot (A
  \log^\alpha(2k²r²) + B)$ for every positive integers $k$
  and~$r$. Observe that $h_r(k) = O(k r^2 \polylog kr)$.
  We will show that graphs whose treewidth is at least $h_r(k)$
  contain $k$ vertex-disjoint models of~$\theta_r$.
  For every positive integers $r$ and $k$, if $\tw(G) \geq h_r(k)$ then we have
  \begin{align*}
        \tw(G)&\geq k(2r)^2 \cdot (A \log^\alpha(\tw(G)) + B)&\text{(as we assume $\tw(G)
      < 2k²r²$)}\\
    \frac{\tw(G)}{A \log^\alpha(\tw(G)) + B} &\geq k(2r)^2 & \text{(because
      $A \log^\alpha(\tw(G)) + B$ is positive)}\\
    \frac{\tw(G)}{f_{\mathrm{Prop\ref{p:cc_h}}}(\tw(G))} &\geq k(2r)^2
                                                           &\text{(as
                                                             $\tw(G)$
                                                             is integer)}.
  \end{align*}
 
  Notice that $k$ and $2r$ meet the conditions of
Proposition~\ref{p:cc_h}. Consequently, there is a partition $G_1,
\dots, G_k$ of $G$ into vertex-disjoint subgraphs such that $\forall i
\in \intv{1}{k},\ \tw(G_i) \geq 2r.$ By
Proposition~\ref{p:one_pumpkin}, each of these subgraphs contains a
model of $\theta_r.$ Consequently, $G$ contains $k$ vertex-disjoint
models of $\theta_r,$ as required.
\end{proof}

A very similar proof can be used to show the following lemma, using
Proposition~\ref{p:cc_r}.
\begin{lemma}\label{l:bound2}
  There is a function $h_r(k) = O(k^3r \polylog kr)$ such that, for every
  positive integers $k$ and $r$ and every graph $G$, if $\tw(G) \geq h_r(k)$, then $\pack{\theta_r}(G)
  \geq k$.
\end{lemma}

\subsection{From vertices to edges}
\label{sec:vert-edg}

In this section, we show how an estimation of a vertex-Erdős-Pósa gap
can be derived from the bound on the treewidth obtained in
Section~\ref{sec:bound-tw}.
The proof of the two following lemmas are inspired from the proof
of~\cite[Lemma~2]{FominST11stre}.

\begin{lemma}[adapted from Lemma~2 of~\cite{FominST11stre}]\label{l:v_sep_tri}
  Let $k\geq 3,r$ be two positive integers and $G$ a graph such that
  $\pack{\theta_{r}}(G) = k$. Then $G$ has a~$(\tw(G) +
  1)$-separation triple $(V_1, S, V_2)$ such that
  ${\frac{1}{3} k \leq \pack{\theta_{r}}(\induced{G}{V_1}) \leq \frac{2}{3} k}$.
\end{lemma}

\begin{proof} 
 Let $(T, \mathcal{V})$ be an optimal nice tree decomposition of $G$. For all $t \in
 \vertices(T)$, let~$H_t$ be the subset of $\vertices(G)$ equal
 to~${\paren{\bigcup_{t' \in \desc{T}(t)} V_{t'}} \setminus V_t},$
 that is, informally, the subset of vertices that are in bags
 \emph{below} $V_t$ but not in $V_t.$ We also
 define the function $p \colon \vertices(T) \to \N$ as: $\forall t \in
 \vertices(T),\ p(t) = \pack{\theta_{r}}(\induced{G}{H_t})$, which
 counts the number of vertex-disjoint models of $\theta_{r}$ in the
 subgraph of $G$ induced by $H_t.$

 \begin{remark}\label{r:non_dec}
  The function $p$ is nondecreasing along every path from a vertex
  of $T$ to the root of $T,$ because if a vertex $t' \in
  \vertices(T)$ is a child of a vertex $t \in \vertices(T)$, then
  $H_{t'} \subseteq H_t$, and thus ${\pack{H}(\induced{G}{H_{t'}})
    \leq \pack{H}(\induced{G}{H_{t}})}.$
 \end{remark}

 \begin{remark}\label{r:claims}
   As $T$ is a nice decomposition of $G$, its vertices can be of four
   different types. We make remarks about the value taken by $p$
   depending on the type of the~vertices:
   \begin{description}
     
   \item[Base node $t$:] $p(t) = 0$, because since $t$ has no descendant,
     $H_t = \emptyset$;
     
   \item[Introduce node $t$ with child $t'$:] as the unique element of
     $V_t \setminus V_t'$ cannot appear in the bags of $\desc{T}(t')$
     (by definition of a tree decomposition), $H_t = H_{t'}$ and
     then $p(t) = p(t');$
     
   \item[Forget node $t$ with child $t'$:] $H_t$ contains one vertex
     more than $H_{t'}$ therefore $p(t) - p(t') \in \{0,1\}$; 

   \item[Join node $t$ with children $t_1$ and $t_2$:] $H_{t} = H_{t_1} \cup
     H_{t_2}$, but $H_{t_1}$ and $H_{t_2}$ are disjoint and there is
     no edge between the vertices of $H_{t_1}$ and of $H_{t_2}$ in
     $\induced{G}{H_t}$ (otherwise the set $V_{t_1}=V_{t_2}$ would contain an
     endpoint of this edge, which also belongs to $H_{t_1}$ or
     $H_{t_2},$ and this is contradictory). Thus there is no $\theta_{r}$-model in
     $\induced{G}{H_t}$ that uses (simultaneously) vertices of $H_{t_1}$ and of
     $H_{t_2},$ and therefore $p(t) = p(t_1) + p(t_2).$
   \end{description}
 \end{remark}

 Let $t \in \vertices(T)$ be a node such that $p(t) > \frac{2}{3}k$
 and such that for every child $t'$ of $t$, $p(t') \leq \frac{2}{3} k$. Let us
 make some claims about $t$.\\

\noindent {Claim} 1: such a $t$ exists.

\begin{proofclaim}[of Claim~1]
  The value of $p$ on the root $r$ of $T$ is $k$ (because
  $\induced{G}{H_r} = G$) and according to the previous remark, the
  value of $p$ on base nodes is 0. As $p$ is nondecreasing on a
  path from a base node to the root (see Remark~\ref{r:non_dec}), such a
  vertex $t$ exists.  
\end{proofclaim}

 \noindent {Claim} 2: $t$ is unique.

 \begin{proofclaim}[of Claim~2]
  To show that $t$ is unique, we assume by contradiction that there is
  another $t'\in V(T)$ with $t' \neq t$ and $p(t') > \frac{2}{3}k$,
  and such that for every child $t''$ of $t'$, $p(t'') \leq \frac{2}{3}
  k$. Three cases can occur:

  \begin{itemize}
  \item[(i)] $t'$ is a descendant of $t$. However, $p$ is
    nondecreasing along any path from a vertex to the root
    (Remark~\ref{r:non_dec}) and $p(t') \geq \frac{2}{3}k$, whereas the
    value of $p$ for each child of $t$ is at most $\frac{2}{3}k$: this is
    a contradiction.
  \item[(ii)] $t$ is a descendant of $t'$. The same argument applies
    (symmetric situation).
  \item[(iii)] $t$ and $t'$ are not in the above situations. Let $v$
    be the least common ancestor of $t$ and $t'$. As $p$ is nondecreasing
    along any path from a vertex to the root, the child $v_t$ (respectively
    $v_{t'}$) of $v$ whose $t$ (respectively $t'$) is descendant of should be
    such that $p(v_t) > \frac{2}{3} k$ (respectively $p(v_{t'}) > \frac{2}{3}
    k).$ By definition of $v$, we have $v_t \neq v_{t'}.$ As $v$ is a join
    node, $p(v) = p(v_t) + p(v_{t'}) > \frac{4}{3} k$, which is
    impossible.
  \end{itemize}
 \end{proofclaim}

\noindent {Claim} 3: $t$ is either a forget node or a join node.

\begin{proofclaim}[of Claim~3]
  By definition of~$t,$ the value $p(t)$ is different from the value(s)
  taken by~$p$ over the child(ren) of $t$. This can only occur in the cases of a
  join node or a forget node.
\end{proofclaim}

 We now present a $(\tw(G) + 1)$-separation triple $(V_1, S, V_2)$ of
 $G$ with the required properties.

 \noindent \textbf{Case 1:} $t$ is a forget node with $t'$ as child.

 Let $S = V_{t'}$, $V_1 = H_{t'}$, and~$V_2 = \vertices(G)
 \setminus (V_1 \cup S).$

 \noindent \textbf{Case 2:} $t$ is a join node with $t_1,\ t_2$ as
 children.

 As $\frac{2}{3} k < p(t) = p(t_1) + p(t_2)$ (Remark~\ref{r:claims}),
 there is $i \in \{1,2\}$ such that $p(t_i) \geq \third{k}.$ Let $S
 = V_{t_i}$, $V_1 = H_{t_i}$, and $V_2 = \vertices(G) \setminus
 (V_1 \cup S).$

 In both cases, we have:
 \begin{enumerate}[(i)]
 \item $\frac{1}{3}k \leq \pack{\theta_{r}}(\induced{G}{V_1}) \leq \frac{2}{3}k$ by
  definition of $V_1$ and $t;$\label{i:ii}
 \item $(V_1, S, V_2)$ is a partition of $\vertices(G);$
 \item \label{i:no_edge} there is no edge between a vertex in $V_1$
   and a vertex of $V_2$ (intuitively, $S$ separates $V_1$ and $V_2$); and
 \item $\card{S} \leq \tw(G) + 1$, because $S$ is a bag of an
   optimal tree decomposition of~$G.$
\end{enumerate}

In the case where $t$ is a forget node, the inequality $\frac{1}{3}k \leq
\pack{\theta_{r}}(\induced{G}{V_1})$ of (\ref{i:ii}) holds because
$p(t') \geq p(t) -1 > \frac{2}{3}k-1 \geq \frac{k}{3}$ (\textit{cf.}\ Remark~\ref{r:claims}).
To see why (\ref{i:no_edge}) is true, assume by contradiction that
there are two vertices $u \in V_1$ and $v \in V_2$ such that $\{u,v\}
\in \edges(G).$ Let $s_0 \in \vertices(T)$ be the child of $t$ such
that $S = V_{s_0}$ (\textit{cf.}\ the two different cases above).
By definition of $V_1$ there is a vertex $s_1 \in
\vertices(T)$ of $T$ in $\desc{T}(s_0)$ whose bag
$V_{s_1}$ contains $u.$
By definition of $V_2,$ the vertex $v$ does not belong to the bag
$V_{s_0}$ nor to a bag of a descendant of $s_0$.
Let $s_2$ be a vertex of $T$ containing $u$ and which is,
according to the previous remark, not the bag of a descendant of $s_0$
nor $s_0.$

As $(T, \mathcal{V})$ is a tree decomposition of $G$ and $\{u,v\} \in
\edges(G),$ we have the following:
\begin{itemize}
\item there is a vertex $s \in \vertices(T)$ whose bag contains both $u$ and $v$;
\item the subgraph of $T$ induced by vertices whose bags contain $u$
  (respectively $v$) is connected.
\end{itemize}
Consequently there is a path in $T$ from $s_1$ to $s$ (respectively from
$s_2$ to $s$) each bag of which contains $u$ (respectively $v$). As $s$
is on the (only) path of $T$ linking $s_1$ to $s_2,$ one of $u,v$
belongs to the bag $V_s.$ But this contradicts the fact that
$(V_1, S, V_2)$ is a partition of $\vertices(G).$

\medskip

We conclude that $(V_1, S, V_2)$ is a $(\tw(G) + 1)$-separation triple
of $G$ with the required properties.
\end{proof}

A function $h\colon \R\to\R$ is said to be \emph{superadditive} if for every $x$ and every $y$ in its domain, $f(x)+f(y) \leq f(x+y)$.
\begin{lemma}[Adapted from Lemma~5.4 in~\cite{ChekuriC13larg}]\label{l:small_tw-new}
  Let $h_r$ be a superadditive function such for every graph $G$ and
  every positive integers $r$ and $k$, if $\tw(G) \geq h_r(k)$ then $G
  \contains (k+1) \cdot \theta_r$. For every graph $G$ and every
  positive integer $k$, if $\pack{\theta_r}(G) = k$ then we have
    \[
    \cover{\theta_r}(G) \leq 3\cdot h_r(k) \log (k+1).
    \]
\end{lemma}

\begin{proof}

We proceed by induction on~$k$.

 \noindent \emph{Base case $k = 0$:} Clear.

 \noindent \emph{Induction step $k > 0$:} We assume that the lemma
 holds for every positive integer $k' < k$.
 Let $G$ be a graph such that $\pack{\theta_r}(G) = k$. First, remark
 that $\tw(G) < h_r(k),$ otherwise by definition of $h_r$ we would
 have $\pack{\theta_r}(G) > k.$
 Thus, by Lemma~\ref{l:v_sep_tri} $G$ contains a $h_r(k)$-separation
 triple $(V_1, S, V_2)$ such that $k/3 \leq \pack{\theta_{r}}(\induced{G}{V_1}) \leq  2k/3$. 
 This implies that $k_1,k_2 \leq \floor{2k/3}$, where $k_i = \pack{\theta_{r}}(\induced{G}{V_i})$ for every
 $i\in \{1,2\}$. Also we have $k_1 + k_2 \leq k$ as $G_1$ and $G_2$
 are two vertex-disjoint subgraphs of~$G$.

  The triple~$(V_1, S, V_2)$ is a partition of~$\vertices(G)$, so the
  following holds:
  \begin{align*}
    \cover{\theta_{r}}(G) &\leq \cover{\theta_{r}}(\induced{G}{V_1}) +
                            \cover{\theta_{r}}(\induced{G}{V_2}) + \card{S}\\
                          &\leq \cover{\theta_{r}}(\induced{G}{V_1}) +
                            \cover{\theta_{r}}(\induced{G}{V_2}) + h_r(k)\\
                          &\leq 3\cdot h_r(k_1) \log (k_1+1) + 3\cdot h_r(k_2) \log (k_2+1) + h_r(k) &\text{(induction hyp.)}
\end{align*}
If $k=1$, then $k_1 = k_1 = 0$ and we have $\cover{\theta_{r}}(G) \leq  h_r(k) \leq 3\cdot h_r(k) \log (k+1)$.
We may now assume $k\geq 2$. Observe that in this case, as $k_i \leq \floor{\frac{2}{3} k}$, we get $k_i +1 \leq \frac{3}{4} (k+1)$ for every~$i\in \{1,2\}$.
\begin{align*}
\cover{\theta_{r}}(G) & \leq 3\cdot (h_r(k_1) + h_r(k_2)) \log \left
                        (\frac{3(k+1)}{4} \right) + h_r(k)\\
                          & \leq 3\cdot h_r(k)\log \left (\frac{3(k+1)}{4} \right) + h_r(k)&\text{(superadditivity of $h_r$)}\\
                         & \leq 3\cdot h_r(k) \log (k+1) - 3\cdot
                           \log (4/3) h_r(k) + h_r(k)\\
                          & \leq 3\cdot  h_r(k)\log (k+1).
  \end{align*}
This concludes the proof.
\end{proof}

\begin{corollary}\label{c:esti}
  Let $f_r$ be the  vertex-Erdős-Pósa gap of~$\theta_r$.
  Then we have
  \begin{itemize}
  \item $f_r(k) = O(kr^2 \polylog kr)$;
  \item $f_r(k) = O(k^3r \polylog kr).$
  \end{itemize}
  These estimations follow from Lemmas~\ref{l:bound1}, \ref{l:bound2},
  and~\ref{l:small_tw-new}.
\end{corollary}

The following lemma shows how to translate a vertex-Erdős-Pósa gap
into an edge-Erdős-Pósa gap in the case of $\theta_r.$ The main idea
of the proof is that if the considered graph has small maximum degree,
a small edge-hitting set can be constructed from a small
vertex-hitting set. On the other hand, a big maximum degree forces a
large packing of $\theta_r$-models.

\begin{lemma}\label{l:transl}
  If $f_r$ is the vertex-Erdős-Pósa gap of
  $\theta_r$, then the edge-Erdős-Pósa gap of
  $\theta_r$ is less than $2kr\cdot f_r(k).$
\end{lemma}

\begin{proof}
  Let $G$ be a graph, let $r \geq2$ be an integer and let $f_r$ is the
  vertex-Erdős-Pósa gap of $\theta_r$.
  We want to prove that if $G$ contains less than $k$ edge-disjoint
  models of $\theta_r,$ then it has a $\theta_r$-edge-hitting set of
  size less than $2kr \cdot f_r(k).$
 
 According to Remark~\ref{r:2co}, we can assume that $G$ is biconnected. If it is
 not the case, we consider its biconnected components separately
 (if it has no biconnected component then the lemma is trivial).

 First, remark $\maxdeg(G) < 2kr,$ otherwise by Lemma~\ref{l:bigdg}
 $G$ would contain at least $k$ edge-disjoint $\theta_r$-models.
 
 Notice that if $G$ does not contain $k$ edge-disjoint
$\theta_r$-models, it does not contain $k$ vertex-disjoint
$\theta_r$-models either. Consequently, there is a set $X \subseteq
\vertices(G)$ meeting every $\theta_r$ model of $G$ and such that
$\card{X} \leq f_r(k).$ Let us consider the set $Y \subseteq \edges(G)$
of edges incident to vertices of $X,$ \ie~$Y = \{\{u,v\} \in
\edges(G),\ u \in X\}.$ Remark that as $\maxdeg(G) < 2kr,$ we have
$\card{Y} \leq 2kr \cdot f_r(k).$ Now, assume that there is a
$\theta_r$-model in $G$ not having edges in~$Y.$ None of its vertices
is in $X,$ which is contradictory. So $Y$ is a $\theta_r$-edge hitting
set of the required size. This concludes the proof.
\end{proof}

\begin{corollary}\label{c:cubic}
  An edge-gap of $O(k³r³)$ for $\theta_r$ can be derived
  from Proposition~\ref{p:vep}.
\end{corollary}

\begin{proof}[Proof of Theorem~\ref{main}]
  It follows from the application of Lemma~\ref{l:transl} to the estimations
  of the vertex-Erdős-Pósa gap of $\theta_r$ given in Corollary~\ref{c:esti}.
\end{proof}

\section{Further research}
\label{sec:fr}

The main question, initiated in this paper, is whether for every
planar graph $J$, the class ${\cal M}(J)$ satisfies this edge variant
of the Erdős-Pósa property. As for the vertex version, it is
easy to see that the planarity of $J$ is necessary. For instance, if
$J=K_5$, consider as graph $G$ an $n$-vertex toroidal wall, which is a
3-regular graph embeddable in the torus that contains $K_5$ as a
minor. One can check that $G$ does not contain two edge-disjoint
models of $K_5$, but $\Omega(\sqrt{n})$ edges of $G$ are needed in
order to hit all its $K_5$-models.

Moreover, a second question is:
when this property holds, does it hold with a polynomial gap for all
graphs? Also, finding lower bounds on this gap for specific graphs is
another interesting and complementary question. Let us mention that,
as it is the case for the vertex version
(see~\cite{ErdosP65,FioriniJW12excl}), for any non-acyclic planar
graph $J$ for which the edge variant of the Erdős-Pósa property holds
for ${\cal M}(J)$, we have that $f_{{\cal M}(J)}(k) = \Omega(k \log
k)$. Indeed, let $G$ be an $n$-vertex cubic graph with treewidth
$\Omega(n)$ and girth $\Omega(\log n)$ (such graphs are well-known to
exist). Since $J$ is planar, the treewidth of any graph excluding $J$
as a minor is bounded by a constant~\cite{RobertsonS86GMV}, hence any
set of edges of $G$ meeting all models of $J$ has size $\Omega(n)$ (as
the removal of an edge may decrease the treewidth by at most two). On
the other hand, since $J$ contains a cycle and the girth of $G$ is
$\Omega(\log n)$, any model of $J$ in $G$ contains $\Omega(\log n)$
edges (assuming that $J$ does not have isolated vertices), and
therefore $G$ contains $O(n / \log n)$ edge-disjoint models of $J$
(here we have used that the degree of $G$ is bounded), easily implying
that $f_{{\cal M}(J)}(k) = \Omega(k \log k)$. In particular, it holds
that $f_{{\cal M}(\theta_{r})}(k) = \Omega(k \log k)$ for any $r \geq
2$, so a first avenue for further work in this direction is to
optimize the gap function $f_{{\cal M}(\theta_{r})}(k)$ given in
Theorem~\ref{main}.

Finally, when the graphs $G$ (in which the packings or
coverings are taken) are restricted to classes of bounded degree, the
proof of Lemma~\ref{l:transl} can easily be adapted to prove that the
bound of the vertex version also holds for the edge version.


\end{document}